\documentclass[12pt]{article}
\RequirePackage{amsthm,amsmath,amsfonts,amssymb}
\RequirePackage[numbers,sort&compress]{natbib}
\RequirePackage[colorlinks,citecolor=blue,urlcolor=blue]{hyperref}
\RequirePackage{graphicx}
\usepackage{bm}
\usepackage{booktabs}

\usepackage{xr}
\newcommand{\blind}{1}
\makeatletter
\def\singlespace{\def\baselinestretch{1}\@normalsize}

\theoremstyle{plain}

\theoremstyle{remark}

\newcommand{\bX}{\mathbf{X}}
\newcommand{\bY}{\mathbf{Y}}
\newcommand{\bx}{\mathbf{x}}
\newcommand{\bz}{\mathbf{z}}
\newcommand{\by}{\mathbf{y}}

\newcommand{\E}{{\rm E}}

\newcommand{\bR}{\mathbf{R}}
\newcommand{\bG}{\mathbf{G}}

\newcommand{\be}{\mathbf{e}}

\newcommand{\bv}{\mathbf{v}}

\newcommand{\bM}{\mathbf{M}}

\newcommand{\bGamma}{\bm{\Gamma}}

\newcommand{\bSig}{\bm{\Sigma}}

\newcommand{\bB}{{\bf B}}

\newcommand{\bS}{{\bf S}}
\newcommand{\bU}{{\bf U}}
\newcommand{\bV}{{\bf V}}

\newcommand{\bI}{{\bf I}}
\newcommand{\diag}{{\rm diag}}

\renewcommand{\(}{\left(}
\renewcommand{\)}{\right)}

\allowdisplaybreaks

\numberwithin{equation}{section}  % number the equations by section

\newtheorem{thm}{Theorem}[section]

\newtheorem{rem}{Remark}[section]
\newtheorem{cor}{Corollary}[section]

\allowdisplaybreaks      %%
\begin{document}

\def\spacingset#1{\renewcommand{\baselinestretch}%
{#1}\small\normalsize} \spacingset{1}

\if1\blind
{
\title{Inference for Spiked Eigenstructure under Generalized Covariance and Correlation Models}
  \author{Yanqing Yin \\
    School of Statistics and Data Science \\
			Nanjing Audit University
    \\
    Wang Zhou\\
    Department of Statistics and Data Science\\  National University of Singapore} 
  \maketitle
} \fi

\if0\blind
{
  \bigskip
  \bigskip
  \bigskip
  \begin{center}
    {\LARGE\bf  Inference for Spiked Eigenstructure under Generalized Covariance and Correlation Models}
\end{center}
  \medskip
} \fi

%%%%%%%%%%%%%%%%%%%%%%%%%%%%%%%%%%%%%%%%%%%%%%
%% Addresses                                %%
%%%%%%%%%%%%%%%%%%%%%%%%%%%%%%%%%%%%%%%%%%%%%%

\bigskip

\begin{abstract}
In high-dimensional principal component analysis, important inferential targets include both leading spikes and the associated principal eigenspaces. Such problems arise naturally in high-dimensional factor models, where leading principal directions are interpreted as dominant loading directions and spike magnitudes reflect the strength of the corresponding common factors. We study inference based on the sample covariance matrix $\bS$ and the sample correlation matrix $\widehat{\bR}$ under generalized spiked models with arbitrary bulk spectrum. We establish almost sure limits and central limit theorems for spiked sample eigenvalues, and derive asymptotic distributions for functionals of sample spiked eigenspaces. Building on this theory, we develop procedures for one-sample inference for benchmark principal directions and for two-sample comparison of leading spike strengths across populations. Even in the covariance setting, our results substantially extend the existing literature by allowing a non-identity bulk structure. A real-data analysis on stock returns further illustrates the practical relevance of the proposed procedures, showing that covariance-based and correlation-based PCA can lead to markedly different conclusions.
\end{abstract}

\noindent%
\vfill

\newpage
\spacingset{1.9}

\section{Introduction}
\subsection{Spiked eigenstructure inference in high-dimensional PCA}
Principal component analysis (PCA) is a fundamental tool in multivariate statistics for dimension reduction, feature extraction, and exploratory data analysis. In modern applications, however, statistical interest often extends beyond the leading eigenvalues themselves to the associated principal eigenspaces and other functionals of the spiked eigenstructure. These objects play a central role in problems such as signal extraction, factor modeling, and low-dimensional representation of high-dimensional data.

Let $\bY=(\by_1,\ldots,\by_n)$ be a $p\times n$ data matrix, where $\by_1,\ldots,\by_n\in\mathbb{R}^p$ are independent observations from a $p$-dimensional population with covariance matrix $\bSig$. PCA is classically performed through the eigendecomposition of $\bSig$ and its sample analogue
\[
\bS=\frac{1}{n}\sum_{j=1}^n \by_j\by_j^\top=\frac{1}{n}\bY\bY^\top.
\]
The eigenvalues of $\bSig$ describe the strength of the principal components, while the associated eigenvectors determine their dominant projection directions.

In the high-dimensional regime, where both $p$ and $n$ grow and $p/n$ converges to a positive constant, the sample covariance matrix is no longer a consistent estimator of $\bSig$ in spectral norm, and the sample eigenstructure may differ substantially from its population counterpart. A natural framework for studying this phenomenon is the spiked model introduced by \citet{Johnstone01D}, in which a finite number of population eigenvalues are separated from the bulk spectrum. The phase transition for spiked sample eigenvalues was established by \citet{BaikS06E}, and a substantial literature has since been developed on the asymptotic behavior of spiked eigenvalues and eigenvectors; see, among others, \citet{BaikB05P,Paul07A,johnstone2009consistency,LedoitP10E,TonyCai2020,Jiang2021b,Bao2022,zhangzheng2022}.

Despite this progress, an important gap remains. In many inferential problems, the relevant statistical target is not the raw eigenvector itself, but rather a projector-type eigenspace functional or a spike-related quantity. Such projector-type functionals are free of sign ambiguity and remain meaningful in the presence of multiple spikes, while spike magnitudes directly quantify the strength of dominant factors or common dependence structures. A general asymptotic theory for these quantities is therefore essential for principled inference in high-dimensional PCA.

Beyond one-sample recovery questions, it is also natural in many applications to compare dominant spectral structure across different populations, groups, or time periods. Such two-sample questions arise, for example, when one wishes to assess whether the strength or geometry of the dominant dependence structure remains stable across regimes. This further highlights the need for an inferential framework that can accommodate both eigenspace-type targets and spike-related quantities in high dimensions.

In many modern applications, these questions also admit a natural interpretation through high-dimensional factor models, where leading eigenvectors correspond to dominant loading directions and spike magnitudes quantify the strength of the associated common factors. 

\subsection{Covariance PCA, correlation PCA, and normalization}

In practice, PCA is commonly performed either on the sample covariance matrix $\bS$ or on the sample correlation matrix
\[
\widehat{\bR}=\diag^{-1/2}(\bS)\bS\diag^{-1/2}(\bS).
\]
The former reflects the raw variation structure of the data, whereas the latter removes the influence of heterogeneous marginal scales and is therefore widely used when the variables are measured in different units or exhibit substantially different variances.

From a high-dimensional perspective, correlation-based PCA is substantially more delicate than covariance-based PCA. The reason is that $\widehat{\bR}$ depends nonlinearly on the diagonal entries of $\bS$, so the normalization step induces a nontrivial global dependence structure even when the underlying components are independent. As a consequence, although the first-order spectral behavior of $\widehat{\bR}$ can often be described through the corresponding population correlation matrix, its second-order behavior is much harder to analyze.

Existing works have shown that normalization may alter the fluctuation behavior of both eigenvalues and eigenvectors; see, for example, \citet{ElKaroui09C,GaoH17H,Yin,Yin2022a,Morales-Jimenez2021}. In particular, \citet{Morales-Jimenez2021} studied spiked sample correlation matrices under a special model with identity bulk structure and independent spiked and non-spiked parts, and showed that normalization leaves the first-order limits of the spiked eigenstructure unchanged while modifying its second-order behavior. These findings suggest that normalization has a genuinely nontrivial inferential effect in high dimensions.

This leads naturally to the following question: how does normalization affect principal direction and eigenspace inference under a generalized spiked model? The present paper addresses this question by studying the asymptotic behavior of the spiked eigenstructure of both $\bS$ and $\widehat{\bR}$ within a unified framework. Our goal is not only to study one-sample recovery and testing problems, but also to provide a basis for comparing dominant spectral structure across populations in covariance- and correlation-based PCA.

\subsection{Related work and the gap in the literature}

The asymptotic theory of spiked eigenvalues and eigenvectors has been extensively studied in the covariance setting. In particular, a substantial literature has established the limiting behavior of sample spiked eigenvalues, eigenvectors, and related projection statistics under various spiked covariance models. Nevertheless, much of the available theory is still developed under relatively restrictive bulk structures, most notably the identity bulk case or closely related simplified models. Compared with covariance-based PCA, the literature on correlation-based PCA remains much thinner, especially under generalized bulk spectrum and without idealized independence assumptions between the spiked and non-spiked parts.

Two limitations of the existing literature are particularly relevant for statistical inference. First, although projection-type statistics have been studied in the covariance setting, existing results are still largely tied to special bulk structures. For example, \citet{Bao2022} developed an elegant inferential theory for projections of sample spiked eigenvectors, but their analysis is carried out under a model with identity bulk. By contrast, our results allow for a general non-identity bulk spectrum, which is important for many realistic high-dimensional covariance structures. Second, in the correlation setting, available results are much more limited under generalized bulk models. In particular, a general inferential theory for projector-type eigenspace functionals and spike-related quantities under correlation-based PCA is still lacking.

The present paper addresses both limitations. We develop a unified asymptotic theory for spiked covariance and correlation PCA under generalized bulk structure, with emphasis on inferential targets that are directly meaningful in applications. This perspective is especially relevant in high-dimensional factor analysis, where principal directions represent dominant loading directions and spike magnitudes quantify the strength of the corresponding common factors. There are also related results in stronger signal regimes; see, for example, \citet{Fan2026test}. However, that line of work focuses on much stronger spikes, typically of order $p$, whereas our theory applies to general separated spikes outside the bulk and therefore covers a substantially broader range of signal strengths.

\subsection{Our contributions and organization}

Our contributions are fourfold.

First, we develop a unified framework for spiked eigenstructure inference based on both the sample covariance matrix $\bS$ and the sample correlation matrix $\widehat{\bR}$ under generalized spiked models with arbitrary bulk spectrum. On the covariance side, this extends existing inferential results beyond the identity-bulk setting. On the correlation side, it provides a comparable framework for normalized data under general bulk structure.

Second, within this framework, we establish almost sure limits and central limit theorems for spiked sample eigenvalues, and derive asymptotic distributions for projector-type functionals of sample spiked eigenspaces. These results form the main theoretical foundation for inference on spiked eigenstructure in high dimensions.

Third, we translate this asymptotic theory into feasible statistical procedures. In particular, we develop one-sample inference for benchmark principal directions and two-sample inference for comparing leading spike strengths across populations.

Fourth, we clarify the inferential effect of normalization in high-dimensional PCA. We show that covariance-based and correlation-based PCA may share the same first-order inferential targets, while exhibiting substantially different second-order fluctuation behavior. Thus, normalization has a direct and essential impact on statistical inference, rather than merely serving as a preprocessing step.

The remainder of the paper is organized as follows. Section \ref{model} introduces the model and inferential targets, Section \ref{sectheory} develops the asymptotic theory, Section \ref{secapp} presents the statistical applications, and Section \ref{secsimu} reports numerical results. Technical proofs and additional supporting materials are deferred to the supplementary material.

\section{Model setup and inferential targets}\label{model}

This section introduces the generalized spiked covariance and correlation models studied throughout the paper and defines the main inferential targets.

\subsection{Generalized spiked covariance and correlation models}

Let $\bY=\bGamma \bX$, where
$
\bX=(x_{ij})=(\bx_1,\ldots,\bx_n)
$
has i.i.d.\ entries with mean $0$ and variance $1$. The population covariance matrix is
$
\bSig=\bGamma\bGamma^\top.
$
Define
$
\bG=[\diag(\bSig)]^{-1/2}\bGamma,
$
so that the population correlation matrix is
$
\bR=\bG\bG^\top.
$

The corresponding covariance- and correlation-based PCA procedures are based on
\[
\bS=\frac{1}{n}\bY\bY^\top,\qquad
\widehat{\bR}=\diag^{-1/2}(\bS)\bS\diag^{-1/2}(\bS).
\]

Write the singular value decomposition of $\bG$ as
$
\bG=\bV\diag[\boldsymbol\Lambda_s^{1/2},\boldsymbol\Lambda_b^{1/2}]\bU^\top,
$
where $\boldsymbol\Lambda_s$ contains the spiked eigenvalues and $\boldsymbol\Lambda_b$ contains the bulk eigenvalues. Let
$
\alpha_1>\cdots>\alpha_K
$
be the distinct population spikes of $\bR$, with multiplicities $m_k$, $k=1,\ldots,K$, satisfying
$
\sum_{k=1}^K m_k=M,
$
where $M$ is fixed. Let $\bV_s$ and $\bV_b$ be the matrices formed by the first $M$ and the remaining $p-M$ columns of $\bV$, respectively. Similarly, let $\bU_s$ and $\bU_b$ be formed from $\bU$, and define
$
\bG_b=\bV_b\boldsymbol\Lambda_b^{1/2}\bU_b^\top.
$

Denote the eigenvalues and eigenvectors of $\widehat{\bR}$ by
\[
\lambda_1(\widehat{\bR}),\ldots,\lambda_p(\widehat{\bR}),
\qquad
\widehat z_1,\ldots,\widehat z_p,
\]
and those of $\bS$ by
\[
\lambda_1(\bS),\ldots,\lambda_p(\bS),
\qquad
\widehat z_{\bSig,1},\ldots,\widehat z_{\bSig,p}.
\]

Define
\[
\varphi_{y,H_R}(z)
=
z\left(1+\int \frac{yt}{z-t}\,dH_R(t)\right),
\qquad
\psi_{y,H_R}(z)
=
\frac{\varphi_{y,H_R}(z)}{z}.
\]
We shall write $\varphi(z)$ and $\psi(z)$ for simplicity whenever no confusion arises. Using the standard notation $f'$ to denote the derivative of any function $f$, we have
\[
\varphi_{y,H_R}'(z)
=
1-y\int \frac{t^2}{(z-t)^2}\,dH_R(t),\qquad
\varphi_{y,H_R}''(z)
=
2y\int \frac{t^2}{(z-t)^3}\,dH_R(t),
\]
and
\[
\varphi_{y,H_R}'''(z)
=
-6y\int \frac{t^2}{(z-t)^4}\,dH_R(t).
\]

We impose the following assumptions.

\begin{itemize}
\item {\bf A1 (Moment condition):} The variables $\{x_{ij}\}$ are i.i.d.\ real random variables satisfying
\[
\E(x_{11})=0,\qquad \E(x_{11}^2)=1,\qquad \E(x_{11}^4)=\nu_4+3<\infty.
\]

\item {\bf A2 (Bulk condition):} The bulk correlation matrix
$
\bR_b=\bG_b\bG_b^\top
$
is bounded in spectral norm. Its empirical spectral distribution $H_{R,n}$ converges weakly to a nondegenerate probability distribution $H_R$.

\item {\bf A3 (Comparable scenario):} As $n\to\infty$,
\[
y_n=\frac{p}{n}\to y\in(0,\infty).
\]

\item {\bf A4 (Separable condition):} The spikes $\alpha_1,\ldots,\alpha_K$ lie outside the support of $H_R$ and satisfy
\[
\varphi'_{y,H_R}(\alpha_k)>0,\qquad 1\le k\le K.
\]
Moreover, there exists a constant $d>0$, independent of $p$ and $n$, such that for any $1\le j\neq k\le K$,
\[
\left|\frac{\alpha_k}{\alpha_j}-1\right|>d.
\]
\end{itemize}

Assumptions {\bf A1--A4} will be used for the correlation matrix case. The corresponding covariance-side assumptions are defined analogously by replacing $\bG$ and $\bR$ with $\bGamma$ and $\bSig$, and replacing $H_R$ and $H_{R,n}$ with $H_\Sigma$ and $H_{\Sigma,n}$, respectively. When needed, we refer to them as {\bf A1--A4}$(\bSig)$. For the covariance model, let $H_{\Sigma,n}$ denote the empirical spectral distribution of the bulk covariance matrix and assume that $H_{\Sigma,n}\Rightarrow H_\Sigma$ weakly.

Under these models, our main inferential objects are the sample spiked eigenspaces and their projector-type functionals, together with the spiked sample eigenvalues used in the two-sample problem. Let $I_k$ denote the index set corresponding to the $m_k$ occurrences of the $k$th population spike, $k=1,\ldots,K$, and write
\[
\mathcal I=\{I_1,\ldots,I_K\}.
\]
For the population correlation matrix $\bR$, the eigenspace associated with $\alpha_k$ is spanned by $\{\bV_j:j\in I_k\}$, and the corresponding projector is
\[
\Pi_k=\sum_{j\in I_k}\bV_j\bV_j^\top.
\]
The sample projector based on $\widehat{\bR}$ is
\[
\widehat \Pi_k=\sum_{j\in I_k}\widehat z_j\widehat z_j^\top.
\]
Similarly, for the covariance matrix case, define
\[
\widehat \Pi_{\bSig,k}=\sum_{j\in I_k}\widehat z_{\bSig,j}\widehat z_{\bSig,j}^\top.
\]

For a prescribed unit vector $\mathcal P\in\mathbb R^p$, the corresponding projector-type statistics are
\[
T_{\bR,k}(\mathcal P)=\mathcal P^\top \widehat \Pi_k \mathcal P,
\qquad
T_{\bSig,k}(\mathcal P)=\mathcal P^\top \widehat \Pi_{\bSig,k} \mathcal P.
\]
In the simple-spike case, these reduce to
\[
T_{\bR,k}(\mathcal P)=\big(\mathcal P^\top \widehat z_k\big)^2,
\qquad
T_{\bSig,k}(\mathcal P)=\big(\mathcal P^\top \widehat z_{\bSig,k}\big)^2.
\]
These quantities provide a natural basis for inference on principal directions and eigenspaces while avoiding the sign ambiguity of individual eigenvectors. The spiked sample eigenvalues will serve as the corresponding inferential targets in the two-sample problem studied in Section~4.

\section{Unified asymptotic theory for spiked covariance and correlation PCA}\label{sectheory}
This section presents the main asymptotic results for the spiked eigenstructure of $\widehat{\bR}$ and $\bS$ under the generalized spiked model.

\subsection{Preliminaries from random matrix theory}

To state the main asymptotic results, we introduce several standard quantities from random matrix theory. Let $A_n$ be a $p\times p$ symmetric matrix with eigenvalues
\[
\lambda_1(A_n)\le \cdots\le \lambda_p(A_n).
\]
Its empirical spectral distribution is defined by
\[
F^{A_n}(x)=\frac{1}{p}\sum_{j=1}^p I\(\lambda_j(A_n)\le x\).
\]

Under Assumptions {\bf A1--A4}, the empirical spectral distribution of $\widehat{\bR}$ converges weakly, almost surely, to a nonrandom probability distribution function $F^{y,H_R}$. Its Stieltjes transform $s_{F^{y,H_R}}(z)$ satisfies
\[
s_{F^{y,H_R}}(z)
=
\int \frac{1}{t\(1-y-yzs_{F^{y,H_R}}(z)\)-z}\,dH_R(t),
\qquad z\in\mathbb C^+.
\]
For the covariance matrix case, the corresponding quantities are defined analogously with $H_R$ replaced by $H_\Sigma$. We state the correlation-side results first and then give the covariance-side counterparts through the corresponding substitutions.

\subsection{Spiked eigenvalues of $\widehat{\bR}$ and $\bS$}

We begin with the sample spiked eigenvalues. Under the separability condition, each population spike generates a cluster of sample eigenvalues outside the support of the limiting bulk spectrum. We first state their almost sure limits.

\begin{thm}\label{theigvaluecorr}
	Assume {\bf A1-A4}. For each population spiked eigenvalue $\alpha_k$ with multiplicity $m_k$, where $k=1,\cdots,K$, and the associated sample eigenvalues $\{\lambda_j(\widehat{\bR}), j\in I_k\}$, we have that for all $j\in I_k$, as $n\to\infty$,
	\[
	\lambda_j(\widehat{\bR})/\varphi_{{y_n,H_{R,n}}}(\alpha_k) - 1 \to 0
	\]
	almost surely.
\end{thm}

To describe the second-order fluctuation of the spiked sample eigenvalues, define
\[
\varpi_k=
\left(
\sqrt n\left(\frac{\lambda_j(\widehat{\bR})}{\varphi_{y_n,H_{R,n}}(\alpha_k)}-1\right),\ j\in I_k
\right)^\top.
\]
Its limiting distribution is characterized by the eigenvalues of a Gaussian random matrix whose covariance structure depends on the spike, the bulk spectrum, and the eigenstructure of $\bG$. Let $\mathbb{G}=(\mathcal{G}_{i,j})={\rm Diag}(\mathbb{G}_1,\cdots,\mathbb{G}_K)$, where each $\mathbb{G}_k$ is an $m_k\times m_k$ zero-mean Gaussian random matrix. For $\mathfrak{i}, \mathfrak{j}, \mathfrak{k}, \mathfrak{l} \in I_k$, we have
\begin{align*}
&{\rm Cov}(\mathcal{G}_{\mathfrak{i},\mathfrak{j}}, \mathcal{G}_{\mathfrak{k},\mathfrak{l}})\\
=&\frac{\mathcal{L}^2_0(\alpha_k)\(\be_{\mathfrak{i}}^\top\be_{\mathfrak{k}}\be_{\mathfrak{j}}^\top\be_{\mathfrak{l}}+\be_{\mathfrak{i}}^\top\be_{\mathfrak{l}}\be_{\mathfrak{k}}^\top\be_{\mathfrak{j}}\)}{\varphi'(\alpha_k)}
+\lim_{n\to+\infty}{\nu_4\mathcal{L}^2_0(\alpha_k)}\(\bU_{\mathfrak{i}}\circ \bU_{\mathfrak{j}}\)^\top\(\bU_{\mathfrak{k}}\circ\bU_{\mathfrak{l}}\)\\
&-\nu_4\lim_{n\to+\infty}{\mathcal{L}^2_0(\alpha_k)}\(\(\bU_{\mathfrak{i}}\circ \bU_{\mathfrak{j}}\)^\top(\bG\circ \bG)\(\bV_{\mathfrak{k}}\circ \bV_{\mathfrak{l}}\)+\(\bU_{\mathfrak{k}}\circ \bU_{\mathfrak{l}}\)^\top(\bG\circ \bG)\(\bV_{\mathfrak{i}}\circ \bV_{\mathfrak{j}}\)\)\\
&-4\alpha_k\lim_{n\to+\infty}{\mathcal{L}^2_0(\alpha_k)}(\bV_{\mathfrak{i}}\circ \bV_{\mathfrak{j}})^\top\(\bV_{\mathfrak{k}}\circ \bV_{\mathfrak{l}}\)
+\lim_{n\to+\infty}\mathcal{L}^2_0(\alpha_k)\(\bV_{\mathfrak{i}}\circ\bV_{\mathfrak{j}}\)^\top\mathfrak{R}_{G}^{\nu_4}\(\bV_{\mathfrak{k}}\circ \bV_{\mathfrak{l}}\),
\end{align*}
where the function $$\mathcal{L}_0(z)= \frac{\varphi'(z)}{\psi(z)} = \frac{z \varphi'(z)}{\varphi(z)},\quad \mathfrak{R}_{G}^{\nu_4}=2(\bR \circ \bR) + \nu_4 \left[(\bG \circ \bG)(\bG^\top \circ \bG^\top)\right].$$

\begin{thm}\label{theigvaluemulticltcorr}
	Assume {\bf A1-A4}. For each population spiked eigenvalue $\alpha_k$ with multiplicity $m_k$, where $k=1,\cdots,K$, and the associated sample eigenvalues $\{\lambda_j(\widehat{\bR}), j\in I_k\}$, the distribution of $\varpi_k$ converges weakly to the joint distribution of the $m_k$ eigenvalues of the Gaussian random matrix $\mathbb{G}_k$, as $n\to\infty$.
\end{thm}

When all spikes are simple, the above result reduces to the following multivariate Gaussian limit.

\begin{thm}\label{theigvaluemulticltcorr2}
	Assume {\bf A1-A4}. Consider $K$ simple spiked eigenvalues $\alpha_1,\cdots,\alpha_K$ and the associated sample eigenvalues $\{\lambda_1(\widehat{\bR}),\cdots,\lambda_K(\widehat{\bR})\}$. As $n\to\infty$, the distribution of
	\[
	\left(\sqrt n\left(\lambda_k(\widehat{\bR})/\varphi_{{y_n,H_{R,n}}}(\alpha_k)-1\right), k=1,\cdots,K \right)^\top
	\]
	converges weakly to a $K$-dimensional zero-mean Gaussian vector with covariance matrix $\mathbf{C}=(c_{\mathfrak{i},\mathfrak{j}})$, where
	\[
	c_{\mathfrak{i},\mathfrak{j}}={\rm Cov}(\mathcal{G}_{\mathfrak{i},\mathfrak{i}}, \mathcal{G}_{\mathfrak{j},\mathfrak{j}}).
	\]
\end{thm}

The corresponding covariance-side results take the same general form, with $\bG$ and $\bR$ replaced by $\bGamma$ and $\bSig$, and with $H_R$ replaced by $H_\Sigma$.

\begin{cor}\label{theigvaluecov}
	Assume {\bf A1-A4}$(\bSig)$. For each population spiked eigenvalue $\alpha_k$ with multiplicity $m_k$, $k=1,\ldots,K$, and the associated sample eigenvalues $\{\lambda_j(\bS), j\in I_k\}$, we have the following. For all $j \in I_k$, as $n\to\infty$,
	\[
	\lambda_j(\bS)/\varphi_{{y_n,H_{\Sigma,n}}}(\alpha_k) - 1 \to 0
	\]
	almost surely. Moreover, the distribution of
	\[
	\varpi_k = \left(\sqrt{n}\left(\lambda_j(\bS)/\varphi_{{y_n,H_{\Sigma,n}}}(\alpha_k) - 1\right), j \in I_k\right)^\top
	\]
	converges weakly to the joint distribution of the $m_k$ eigenvalues of the $k$-th block of a Gaussian random matrix $\mathbb{G} = (\mathcal{G}_{i,j}) = \mathrm{Diag}(\mathbb{G}_{1},\ldots,\mathbb{G}_{K})$, where $\mathbb{G}_{k}$ is an $m_k \times m_k$ zero-mean Gaussian random matrix. For $\mathfrak{i}, \mathfrak{j}, \mathfrak{k}, \mathfrak{l} \in I_k$,
	\begin{align*}
		\mathrm{Cov}(\mathcal{G}_{\mathfrak{i},\mathfrak{j}}, \mathcal{G}_{\mathfrak{k},\mathfrak{l}})
		=&\lim_{n\to+\infty}\frac{\mathcal{L}^2_0(\alpha_k)\(\be_{\mathfrak{i}}^\top\be_{\mathfrak{k}}\be_{\mathfrak{j}}^\top\be_{\mathfrak{l}}+\be_{\mathfrak{i}}^\top\be_{\mathfrak{l}}\be_{\mathfrak{k}}^\top\be_{\mathfrak{j}}\)}{\varphi'(\alpha_k)} 
		+ \lim_{n\to+\infty}\nu_4\mathcal{L}^2_0(\alpha_k)\(\bU_{\mathfrak{i}}\circ \bU_{\mathfrak{j}}\)^\top\(\bU_{\mathfrak{k}}\circ \bU_{\mathfrak{l}}\).
	\end{align*}
\end{cor}
\subsection{Principal eigenspace projections of $\widehat{\bR}$ and $\bS$}

We now turn to the principal eigenspace, which is the main inferential object in PCA. As discussed in Section 2, the quantity of direct interest is the projection of a prescribed direction $\mathcal{P}$ onto the sample spiked eigenspace.

For a prescribed unit vector $\mathcal P$, write
\[
\tau_j=\mathcal P^\top \bV_j,\qquad j\in I_k.
\] The following theorem gives the first-order limit and second-order fluctuation of the sample spiked eigenspace projection in an arbitrary direction.

\begin{thm}\label{theigvectorcltcorr}
	Assume {\bf A1-A4}. As $n \to \infty$:
	\begin{itemize}
		\item[(1)] $\mathcal P^\top \widehat\Pi_k \mathcal P - \mathcal{L}_0(\alpha_k)\sum_{j\in I_k}\tau_j^2 \overset{a.s.}{\to} 0;$
		\item[(2)] $\sqrt{n}\left(\mathcal P^\top \widehat\Pi_k \mathcal P - \mathcal{L}_0(\alpha_k)\sum_{j\in I_k}\tau_j^2\right) \overset{D}{\to} N(0,\sigma_k^2),$
	\end{itemize}
	where
	\[
	\sigma_k^2 = \sum_{j=1}^7\mathcal{V}_{j,j}^{(k)} + 2\sum_{1 \leq j < \ell \leq 7}\mathcal{V}_{j,\ell}^{(k)}.
	\]
\end{thm}

%{\color{red}Taking $\mathcal{P}$ as $\be$, consider the simple spike case, one find that the sample eigenvector also obey a ``signal-plus-noise" structure. Say it is the length contracted true spiked population eigenvector plus a delocalized random noise vector with the same order entries.}

The variance $\sigma_k^2$ admits an explicit decomposition into seven groups of terms, corresponding to the fluctuation contributions from the spiked resolvent part, spike interaction, bulk interaction, and the additional normalization terms. Since the complete expressions are lengthy and would interrupt the flow of the main exposition, we defer the explicit formulas for
$
\mathcal{V}_{j,\ell}^{(k)}, 1\le j\le \ell\le 7,
$
to the supplementary material.

In the simple-spike case, the result takes a more transparent form. Taking $\mathcal{P}=\bV_k$, we obtain the following asymptotic distribution for the alignment between the sample and population principal directions. For notational simplicity, the auxiliary quantities $\mathcal L_2(\alpha_k)$ and $\mathcal L_{0'}(\alpha_k)$ are defined in the supplementary material.

\begin{cor}\label{cor1eigvectorcltcorr}
	Assume {\bf A1-A4}. Consider $\bV_k$, the population eigenvector corresponding to a simple spike $\alpha_k$. As $n \to \infty$:
	\begin{itemize}
		\item[(1)] $\left(\bV_k^\top \widehat\bz_k\right)^2 - \mathcal{L}_0(\alpha_k) \overset{a.s.}{\to} 0;$
		\item[(2)] $\sqrt{n}\left(\left(\bV_k^\top \widehat\bz_k\right)^2 - \mathcal{L}_0(\alpha_k)\right) \overset{D}{\to} N(0,\sigma_k^2),$
	\end{itemize}
	where the variance $\sigma_k^2$ is \begin{align*}
		&\lim_{n\to+\infty} 2\alpha_k^2\mathcal{L}_2(\alpha_k) + {\nu_4}\lim_{n\to+\infty} \alpha_k^2\mathcal{L}^2_{0'}(\alpha_k)\left(\bU_k\circ\bU_k\right)^\top \left(\bU_k\circ\bU_k\right) \\
		& + \lim_{n\to+\infty} \alpha_k^2\mathcal{L}^2_{0'}(\alpha_k)\left(\bV_k\circ\bV_k\right)^\top \mathfrak{R}_{G}^{\nu_4} \left(\bV_k\circ\bV_k\right) \\
		& - 2\lim_{n\to+\infty} \alpha_k^2\mathcal{L}^2_{0'}(\alpha_k) \bigg\{ 2\alpha_k \left(\bV_k\circ\bV_k\right)^\top \left(\bV_k\circ\bV_k\right) 
		 + \nu_4 \left(\bV_k\circ\bV_k\right)^\top \left(\bG\circ\bG\right) \left(\bU_k\circ\bU_k\right) \bigg\}.
	\end{align*}
\end{cor}

\begin{rem}
	Corollary \ref{cor1eigvectorcltcorr} is stated for a simple spike in order to facilitate interpretation. The corresponding multiple-spike result follows from Theorem \ref{theigvectorcltcorr} with only the relevant variance components retained.
\end{rem}

\begin{rem}
	If $\mathcal{P}$ is orthogonal to the space spanned by $\{\bV_j:\ j\in I_k\}$, then the corresponding deterministic projection terms vanish. In this case, Theorem \ref{theigvectorcltcorr} implies that $\mathcal{P}^\top \widehat\bz_k \overset{a.s.}{\to} 0$, as $n\to\infty$, and the limiting variance of $\sqrt{n}\left(\mathcal{P}^\top \widehat\bz_k\right)^2$ is zero.
\end{rem}
The corresponding covariance-side results provide a generalized-bulk benchmark for the correlation-side theory.

\begin{cor}\label{theigvectorcltcov}
	Assume {\bf A1-A4}$(\bSig)$. As $n \to \infty$:
	\begin{itemize}
		\item[(1)] $\mathcal P^\top \widehat\Pi_{\bSig,k}\mathcal P - \mathcal{L}_0(\alpha_k)\sum_{j \in I_k}\tau_j^2 \overset{a.s.}{\to} 0;$
		\item[(2)] $\sqrt{n}\left(\mathcal P^\top \widehat\Pi_{\bSig,k}\mathcal P - \mathcal{L}_0(\alpha_k)\sum_{j \in I_k}\tau_j^2\right) \overset{D}{\to} N(0,\sigma_k^2),$
	\end{itemize}
	where
	\[
	\sigma_k^2 = \sum_{j=1}^3\mathcal{V}_{j,j}^{(k)} + 2\sum_{1 \leq j < \ell \leq 3}\mathcal{V}_{j,\ell}^{(k)}.
	\]
The explicit covariance-side variance components are given in the supplementary material.
\end{cor}

\begin{cor}\label{cor1eigvectorcltcov}
	Assume {\bf A1-A4}$(\bSig)$. Consider $\bV_k$, the population eigenvector corresponding to a simple spike $\alpha_k$. As $n \to \infty$:
	\begin{itemize}
		\item[(1)] $\left(\bV_k^\top \widehat\bz_{\bSig,k}\right)^2 - \mathcal{L}_0(\alpha_k) \overset{a.s.}{\to} 0;$
		\item[(2)] $\sqrt{n}\left(\left(\bV_k^\top \widehat\bz_{\bSig,k}\right)^2 - \mathcal{L}_0(\alpha_k)\right) \overset{D}{\to} N(0,\sigma_k^2),$
	\end{itemize}
	where
	\[
	\sigma_k^2 = \lim_{n\to+\infty}2\alpha_k^2\mathcal{L}_2(\alpha_k) + \nu_4\lim_{n\to+\infty}\alpha_k^2\mathcal{L}_{0'}^2(\alpha_k)\left(\bU_k\circ \bU_k\right)^\top\left(\bU_k\circ\bU_k\right).
	\]
\end{cor}

\subsection{Comparison between covariance- and correlation-based PCA}

The preceding results provide a unified asymptotic description of the spiked eigenstructure of both $\bS$ and $\widehat{\bR}$ under a generalized spiked model. They also reveal several important differences between covariance-based and correlation-based PCA in high dimensions.

For the spiked eigenvalues, the covariance-side fluctuation depends on the bulk spectrum of $\bGamma_b$, the spike magnitudes, and the eigenstructure of the spiked part of $\bGamma$. In particular, when $\nu_4=0$, the dependence on the right singular vectors disappears and the fluctuation structure becomes simpler. In the correlation case, however, the asymptotic distribution depends on the full eigenstructure of $\bG$, reflecting the additional randomness induced by diagonal normalization. Thus, although the first-order limit of a sample spiked eigenvalue can still be described through the same deterministic mapping $\varphi$ after replacing $\bSig$ by $\bR$, its second-order behavior cannot be recovered by such a direct substitution.

A similar phenomenon occurs for the spiked eigenspaces. The projector-type limits established above show that the first-order behavior of the sample principal eigenspace has the same general form in both covariance-based and correlation-based PCA. In particular, the limiting amount of captured signal is determined by the factor $\mathcal{L}_0(\alpha_k)$ together with the projection of the target direction onto the corresponding population eigenspace. However, the fluctuation behavior around this limit is substantially different in the two settings. For the covariance matrix, the second-order terms are determined by the bulk structure and the eigenstructure of $\bGamma$. For the correlation matrix, the normalization step introduces additional terms involving the diagonal fluctuation of $\bS$, and these terms contribute essentially to the asymptotic variance of the sample eigenspace projection.

These results also clarify the effect of signal strength on eigenspace recovery. In both settings, Corollaries~\ref{cor1eigvectorcltcorr} and~\ref{cor1eigvectorcltcov} imply that if the spike remains fixed, then there exists a nonvanishing asymptotic angle between the sample principal direction and its population counterpart. By contrast, when the spike diverges, this angle vanishes asymptotically. Indeed, since
\[
\mathcal{L}_0(\alpha_k)
=
\frac{1-y\int \frac{t^2}{(\alpha_k-t)^2}\,dH_R(t)}
{1+y\int \frac{t}{\alpha_k-t}\,dH_R(t)},
\]
we have $\mathcal{L}_0(\alpha_k)\to 1$ whenever $\alpha_k\to\infty$, which implies increasingly accurate eigenspace recovery as the signal strengthens.

The normalization effect can be further illustrated in the special case $\diag(\bSig)=\bI$ with a single simple spike $\alpha_1$. Then $\bG=\bGamma$ and $\bR=\bSig$, so covariance-based and correlation-based PCA share the same first-order target. Nevertheless, their asymptotic variances differ. In the Gaussian case $\nu_4=0$, the additional effect of normalization is determined by
\[
-\lim_{n\to\infty}
\left(
2\alpha_1 (\bV_1\circ \bV_1)^\top (\bV_1\circ \bV_1)
-
(\bV_1\circ \bV_1)^\top (\bR\circ\bR)(\bV_1\circ \bV_1)
\right).
\]
As noted by \citet{Morales-Jimenez2021}, even in special models the sign of this term may in principle be indefinite. Our general theory shows that the effect of normalization is governed jointly by signal strength and the geometry of the leading population direction. In particular, when the leading direction is sufficiently diffuse across coordinates, the additional normalization effect may be small; when the leading component is strong, normalization may even reduce the asymptotic variance.

These observations show that the effect of normalization in high-dimensional PCA is more subtle than the classical low-dimensional intuition. Rather than merely removing scale information, normalization changes the fluctuation structure of the sample spiked eigenstructure, with direct consequences for statistical inference.

\section{Statistical applications}\label{secapp}

We consider two applications of the asymptotic theory developed in Section~3: one-sample inference for a benchmark principal direction and two-sample inference for a leading spike. 

Before introducing the feasible procedures, we note that if $\nu_4\neq 0$, some second-order quantities in the limiting laws depend on the factor matrices $\bGamma$ and $\bG$, rather than only on $\bSig$ and $\bR$. Since a general factorization $\bSig=\bGamma\bGamma^\top$ or $\bR=\bG\bG^\top$ is not unique, these quantities are not identifiable from $\bSig$ or $\bR$ alone unless a specific representative is fixed. Throughout Section~4, we therefore adopt the canonical symmetric square-root convention
\[
\bGamma=\bSig^{1/2},\qquad \bG=\bR^{1/2},
\]
so that all plug-in quantities are well defined as functionals of $\bSig$ or $\bR$.

For implementation, the unknown centering and variance quantities are estimated by plug-in procedures constructed from the sample spikes, empirical bulk eigenvalues, and sample eigendirections. Their consistency is addressed in the supplementary material.

Roughly speaking, in the covariance case, consistency of the full plug-in variance estimator holds. In the correlation case, the only extra difficulty comes from the terms involving $\bG=\bR^{1/2}$ in the non-Gaussian regime $\nu_4\neq 0$. Thus, full consistency of the correlation-side variance estimator follows once the plug-in approximation to $\bR^{1/2}$ is suitably controlled. This holds under standard structural conditions such as sufficiently regular banded or rapidly decaying correlation matrices, which are often satisfied in statistical applications. In particular, when $\nu_4=0$, these terms vanish, so the correlation-side plug-in estimator is also consistent without any additional condition.

\subsection{One-sample inference for a benchmark principal direction}

We begin with the simple-spike case. Let $v_0\in\mathbb R^p$ be a prescribed unit vector. In the covariance and correlation settings, respectively, we consider
\[
H_{0,\bSig}:\ \bV_{\bSig,k}\bV_{\bSig,k}^\top=v_0v_0^\top,
\qquad
H_{0,\bR}:\ \bV_{\bR,k}\bV_{\bR,k}^\top=v_0v_0^\top,
\]
where $\bV_{\bSig,k}$ and $\bV_{\bR,k}$ are the population eigenvectors corresponding to the $k$th simple spike. Since eigenvectors are identifiable only up to sign, the hypotheses are formulated in terms of rank-one projectors.

To unify notation, define
\[
T_{\star,k}(v_0)=
\begin{cases}
\big(v_0^\top \widehat z_{\bSig,k}\big)^2, & \star=\bSig,\\[4pt]
\big(v_0^\top \widehat z_k\big)^2, & \star=\bR,
\end{cases}
\]
where $\star=\bSig$ and $\star=\bR$ correspond to covariance-based and correlation-based PCA, respectively. Let
\[
\mathcal L_{\star,0}(z)=\frac{\varphi_\star'(z)}{\psi_\star(z)}=\frac{z\varphi_\star'(z)}{\varphi_\star(z)},
\]
where $(\varphi_{\bSig},\psi_{\bSig})$ are defined through the covariance-side bulk law $H_\Sigma$, while $(\varphi_\bR,\psi_\bR)$ are defined through the correlation-side bulk law $H_R$. 

For notational convenience, write
\[
\widehat z_{\star,k}=
\begin{cases}
\widehat z_{\bSig,k}, & \star=\bSig,\\
\widehat z_k, & \star=\bR,
\end{cases}
\qquad
\lambda_{\star,j}=
\begin{cases}
\lambda_j(\bS), & \star=\bSig,\\
\lambda_j(\widehat{\bR}), & \star=\bR,
\end{cases}
\]
for $1\le j\le p$, and let $\alpha_{\star,k}$ denote the corresponding population spike in the $\star$-model.

Then, by Corollaries \ref{cor1eigvectorcltcov} and \ref{cor1eigvectorcltcorr}, under $H_{0,\star}$ we have
\[
\sqrt n\Big(T_{\star,k}(v_0)-\mathcal L_{\star,0}(\alpha_{\star,k})\Big)
\stackrel{D}{\longrightarrow}
N\Big(0,\sigma_{\star,k}^2\Big).
\]

In practice, the centering quantity $\mathcal L_{\star,0}(\alpha_{\star,k})$ is unknown because it depends on the spike magnitude and the bulk spectral distribution. For implementation, we estimate the centering term by
\[
\widehat{\mathcal L}_{\star,0,k}
=
-\frac{\widehat U_{\star,k}}
{\lambda_{\star,k}\widehat U_{\star,k}^{(1)}},
\]
where
\[
\widehat U_{\star,k}
=
-\frac{1-(p-K)/n}{\lambda_{\star,k}}
+\frac{1}{n}\sum_{j=K+1}^p\frac{1}{\lambda_{\star,j}-\lambda_{\star,k}},
\]
\[
\widehat U_{\star,k}^{(1)}
=
\frac{1-(p-K)/n}{\lambda_{\star,k}^2}
+\frac{1}{n}\sum_{j=K+1}^p\frac{1}{(\lambda_{\star,j}-\lambda_{\star,k})^2},
\]
and similarly
\[
\widehat U_{\star,k}^{(2)}
=
-\frac{2\{1-(p-K)/n\}}{\lambda_{\star,k}^3}
+\frac{2}{n}\sum_{j=K+1}^p\frac{1}{(\lambda_{\star,j}-\lambda_{\star,k})^3},
\]
\[
\widehat U_{\star,k}^{(3)}
=
\frac{6\{1-(p-K)/n\}}{\lambda_{\star,k}^4}
+\frac{6}{n}\sum_{j=K+1}^p\frac{1}{(\lambda_{\star,j}-\lambda_{\star,k})^4}.
\]
Define further
\[
\widehat c_{\star,k}
=
\frac{\widehat U_{\star,k}\widehat U_{\star,k}^{(2)}}
{\big(\widehat U_{\star,k}^{(1)}\big)^2}
-1
+\frac{\widehat U_{\star,k}}
{\lambda_{\star,k}\widehat U_{\star,k}^{(1)}},
\]
\[
\widehat A_{\star,k}
:=
\frac{\widehat U_{\star,k}^2}
{6\lambda_{\star,k}^4 \big(\widehat U_{\star,k}^{(1)}\big)^4}
\left\{
\lambda_{\star,k}^2\widehat U_{\star,k}^{(3)}
+6\lambda_{\star,k}\widehat U_{\star,k}^{(2)}
+6\widehat U_{\star,k}^{(1)}
\right\},
\]
\[
\widehat B_{\star,k}
:=
\frac{\widehat U_{\star,k}^2\big(\widehat c_{\star,k}\big)^2}
{\lambda_{\star,k}^2\big(\widehat U_{\star,k}^{(1)}\big)^2}.
\]

For plug-in estimation of the eigenvector-dependent quantities, we de-shrink the sample spiked eigenvector by the estimated alignment factor and define
\[
\widehat{\bv}_{\star,k}^{\,\mathrm{ds}}
:=
\frac{\widehat z_{\star,k}}{\sqrt{\widehat{\mathcal L}_{\star,0,k}}}.
\]
All Hadamard-type plug-in quantities are constructed from this de-shrunk vector.

For the covariance case, we define
\[
\widehat\kappa_{\bSig,k}
=
\bigl(
\widehat{\bv}_{\bSig,k}^{\,\mathrm{ds}}
\circ
\widehat{\bv}_{\bSig,k}^{\,\mathrm{ds}}
\bigr)^\top
\bigl(
\widehat{\bv}_{\bSig,k}^{\,\mathrm{ds}}
\circ
\widehat{\bv}_{\bSig,k}^{\,\mathrm{ds}}
\bigr),
\qquad
\widehat\rho_{\bSig,k}
=
\widehat\gamma_{\bSig,k}
=
\widehat\chi_{\bSig,k}
=
0.
\]

For the correlation case, we define
\[
\widehat\kappa_{\bR,k}
=
\bigl(
\widehat{\bv}_{\bR,k}^{\,\mathrm{ds}}
\circ
\widehat{\bv}_{\bR,k}^{\,\mathrm{ds}}
\bigr)^\top
\bigl(
\widehat{\bv}_{\bR,k}^{\,\mathrm{ds}}
\circ
\widehat{\bv}_{\bR,k}^{\,\mathrm{ds}}
\bigr),
\]
\[
\widehat\rho_{\bR,k}
=
\widehat\kappa_{\bR,k},
\]
\[
\widehat\gamma_{\bR,k}
=
\bigl(
\widehat{\bv}_{\bR,k}^{\,\mathrm{ds}}
\circ
\widehat{\bv}_{\bR,k}^{\,\mathrm{ds}}
\bigr)^\top
(\widehat\bR^{1/2}\circ \widehat\bR^{1/2})
\bigl(
\widehat{\bv}_{\bR,k}^{\,\mathrm{ds}}
\circ
\widehat{\bv}_{\bR,k}^{\,\mathrm{ds}}
\bigr),
\]
\[
\widehat\chi_{\bR,k}
=
\bigl(
\widehat{\bv}_{\bR,k}^{\,\mathrm{ds}}
\circ
\widehat{\bv}_{\bR,k}^{\,\mathrm{ds}}
\bigr)^\top
\widehat{\mathfrak R}_{\bR}^{\nu_4}
\bigl(
\widehat{\bv}_{\bR,k}^{\,\mathrm{ds}}
\circ
\widehat{\bv}_{\bR,k}^{\,\mathrm{ds}}
\bigr),
\]
where
\[
\widehat{\mathfrak R}_{\bR}^{\nu_4}
=
2(\widehat\bR\circ \widehat\bR)
+\nu_4\big[(\widehat\bR^{1/2}\circ \widehat\bR^{1/2})(\widehat\bR^{1/2}\circ \widehat\bR^{1/2})^\top\big].
\]
To unify notation, we write
\[
(\widehat\kappa_{\star,k},\widehat\rho_{\star,k},\widehat\gamma_{\star,k},\widehat\chi_{\star,k})
\]
for the corresponding plug-in quantities in the $\star$-model, where $\star=\bSig$ or $\bR$.

Using these quantities, define
\[
\widehat\sigma_{\lambda,\star,k}^2
=
\frac{2}{\lambda_{\star,k}^2\,\widehat U_{\star,k}^{(1)}}
+
\frac{
4\widehat U_{\star,k}\widehat\rho_{\star,k}
+
\widehat U_{\star,k}^{\,2}
\Big(
\nu_4\widehat\kappa_{\star,k}
-2\nu_4\widehat\gamma_{\star,k}
+\widehat\chi_{\star,k}
\Big)
}{
\lambda_{\star,k}^2\big(\widehat U_{\star,k}^{(1)}\big)^2
},
\]
\[
\widehat\sigma_{\star,k}^2
=
2\widehat A_{\star,k}
+
\widehat B_{\star,k}
\Big(
\nu_4\widehat\kappa_{\star,k}
+\frac{4\widehat\rho_{\star,k}}{\widehat U_{\star,k}}
-2\nu_4\widehat\gamma_{\star,k}
+\widehat\chi_{\star,k}
\Big),
\]
\[
\widehat\eta_{\star,k}
=
-\frac{
\widehat U_{\star,k}\widehat U_{\star,k}^{(2)}
}{
\lambda_{\star,k}^2\big(\widehat U_{\star,k}^{(1)}\big)^3
}
-
\frac{
2\widehat U_{\star,k}
}{
\lambda_{\star,k}^3\big(\widehat U_{\star,k}^{(1)}\big)^2
}
+
\frac{
\widehat U_{\star,k}^{\,2}\widehat c_{\star,k}
}{
\lambda_{\star,k}^2\big(\widehat U_{\star,k}^{(1)}\big)^2
}
\left(
-\frac{4\widehat\rho_{\star,k}}{\widehat U_{\star,k}}
+2\nu_4\widehat\gamma_{\star,k}
-\nu_4\widehat\kappa_{\star,k}
-\widehat\chi_{\star,k}
\right),
\]
and
\[
\widehat{\widetilde\sigma}_{\star,k}^2
=
\widehat\sigma_{\star,k}^2
+
\widehat c_{\star,k}^{\,2}\widehat\sigma_{\lambda,\star,k}^2
+
2\widehat c_{\star,k}\widehat\eta_{\star,k}.
\]

The explicit derivation of these estimators and their population counterparts is deferred to the supplementary material. If $\widehat{\widetilde\sigma}_{\star,k}^2$ is consistent for the corresponding asymptotic variance, then
\[
Z_{\star,k}^{(1)}
:=\frac{
\sqrt n\Big(T_{\star,k}(v_0)-\widehat{\mathcal L}_{\star,0,k}\Big)
}{
\widehat{\widetilde\sigma}_{\star,k}
}
\overset{D}{\longrightarrow}
N(0,1).
\]

Therefore, an asymptotic level-$\beta$ two-sided test rejects $H_0$ whenever
\[
|Z_{\star,k}^{(1)}|>z_{1-\beta/2},
\]
where $z_{1-\beta/2}$ denotes the $(1-\beta/2)$ quantile of the standard normal distribution.

More generally, for a prescribed unit vector $\mathcal P$, one may consider the projector-type statistics
\[
T_{\bSig,k}(\mathcal P)=\mathcal P^\top \widehat\Pi_{\bSig,k}\mathcal P,
\qquad
T_{\bR,k}(\mathcal P)=\mathcal P^\top \widehat\Pi_k\mathcal P.
\]
The same inferential principle continues to apply, although the explicit corrected variances become more cumbersome.

\subsection{Two-sample inference for a spiked eigenvalue}

We next consider a two-sample problem. Let $n_1$ and $n_2$ be two sample sizes, and assume that
\[
\frac{p}{n_i}\to y_i\in(0,\infty),\qquad i=1,2,
\]
so that both samples are of the same high-dimensional order as the dimension $p$. For $i=1,2$, let $\lambda_{i,\star,k}$ denote the $k$th sample spiked eigenvalue from the $i$th sample under the $\star$-model, where $\star=\bSig$ or $\star=\bR$, and let $\alpha_{i,\star,k}$ be the corresponding population spike. We consider the null hypothesis
\[
H_0:\ \alpha_{1,\star,k}=\alpha_{2,\star,k}.
\]

Here we focus on two-sample inference for spiked eigenvalues rather than for spiked eigenvectors. The reason is that, in the generic high-dimensional regime, the sample eigenvector still has a non-negligible angle from its population counterpart unless the spike is sufficiently strong. Consequently, a direct two-sample comparison of sample eigenvectors is typically dominated by random overlap between the corresponding estimation errors, making it difficult to construct a generally stable and explicit test statistic. By contrast, the spiked eigenvalue admits a cleaner two-sample Wald-type comparison.

This problem is particularly interpretable in the correlation setting. Since the diagonal entries of the population correlation matrix are all equal to one, the spikes of $\bR$ directly quantify the relative contribution of the leading factors to the overall correlation structure. Thus, in the correlation case, testing
\[
H_0:\ \alpha_{1,\bR,k}=\alpha_{2,\bR,k}
\]
amounts to testing whether the $k$th principal factor has the same contribution across the two populations.

For each sample $i=1,2$, by the one-sample eigenvalue CLT,
\[
\sqrt{n_i}\left(
\frac{\lambda_{i,\star,k}}{\varphi_\star(\alpha_{i,\star,k})}-1
\right)
\overset{D}{\longrightarrow}
N\bigl(0,\sigma_{\lambda,i,\star,k}^2\bigr),
\]
where $\sigma_{\lambda,i,\star,k}^2$ is defined as in the one-sample case, with all quantities computed from the $i$th sample.

To compare the population spikes directly, let
\[
\widehat\alpha_{i,\star,k}
=
-\frac{1}{\widehat U_{i,\star,k}},
\qquad i=1,2,
\]
where $\widehat U_{i,\star,k}$ is defined exactly as in the one-sample case, but with $n$, $\lambda_{\star,k}$, and the associated bulk eigenvalues replaced by $n_i$, $\lambda_{i,\star,k}$, and the bulk eigenvalues from the $i$th sample, respectively. Likewise, let $\widehat U_{i,\star,k}^{(1)}$ and $\widehat\sigma_{\lambda,i,\star,k}^2$ denote the corresponding one-sample plug-in quantities from sample $i$.

We estimate $\varphi_\star'(\alpha_{i,\star,k})$ by
\[
\widehat\varphi_{i,\star,k}'
=
\frac{1}{\widehat\alpha_{i,\star,k}^{\,2}\,\widehat U_{i,\star,k}^{(1)}}
=
\frac{\widehat U_{i,\star,k}^{\,2}}{\widehat U_{i,\star,k}^{(1)}}.
\]
Then, by the delta method,
\[
\sqrt{n_i}\bigl(
\widehat\alpha_{i,\star,k}-\alpha_{i,\star,k}
\bigr)
\overset{D}{\longrightarrow}
N\bigl(0,\tau_{i,\star,k}^2\bigr),
\]
where
\[
\tau_{i,\star,k}^2
=
\frac{\varphi_\star(\alpha_{i,\star,k})^2}
{\varphi_\star'(\alpha_{i,\star,k})^2}\,
\sigma_{\lambda,i,\star,k}^2.
\]
Accordingly, we define the plug-in estimator
\[
\widehat\tau_{i,\star,k}^2
=
\frac{\lambda_{i,\star,k}^2}
{\widehat\varphi_{i,\star,k}'^{\,2}}\,
\widehat\sigma_{\lambda,i,\star,k}^2.
\]

Under the same regularity conditions ensuring the consistency of the one-sample plug-in estimators, $\widehat\tau_{i,\star,k}^2$ is consistent for $\tau_{i,\star,k}^2$. Since the two samples are independent, under $H_0$ we have
\[
Z_{\star,k}^{(2)}
:=
\frac{
\widehat\alpha_{1,\star,k}-\widehat\alpha_{2,\star,k}
}{
\sqrt{
\widehat\tau_{1,\star,k}^2/n_1
+
\widehat\tau_{2,\star,k}^2/n_2
}
}
\overset{D}{\longrightarrow}
N(0,1).
\]

Therefore, an asymptotic level-$\beta$ two-sided test rejects $H_0$ whenever
\[
|Z_{\star,k}^{(2)}|>z_{1-\beta/2},
\]
where $z_{1-\beta/2}$ denotes the $(1-\beta/2)$ quantile of the standard normal distribution.

\section{Numerical studies}\label{secsimu}

This section evaluates the finite-sample performance of the inferential procedures developed in Section~4, including one-sample benchmark-direction testing and two-sample spike comparison. Unless otherwise stated, the nominal level is fixed at $\beta=0.05$, and each rejection probability is estimated from $10{,}000$ Monte Carlo replications.

In all simulations, the data matrix is generated as $
\bY=\bM^{1/2} \bX,
$ where the entries of $\bX$ are i.i.d.\ with mean zero and variance one. We consider two innovation distributions: the standard Gaussian law and the centered Gamma law Gamma$(4,0.5)-2$, whose variance equals one and whose excess kurtosis is $1.5$. These designs are chosen to reflect two representative scenarios for the bulk structure: a relatively simple homogeneous dependence pattern and a more heterogeneous bulk-plus-spike configuration. Together with the Gaussian and non-Gaussian innovations, they provide a basic check of the robustness of the proposed procedures across model settings.

We consider two representative population designs. In the equicorrelated design, the population matrix takes the form
$
\bM=(1-\rho_0)\bI+\rho_0\mathbf1\mathbf1^\top,
$
and in the Toeplitz-type design,
$
\bM=\bB_{\rm bulk}+\theta uu^\top,
\qquad
\bB_{\rm bulk}=(a^{|i-j|})_{1\le i,j\le p},
$
where $u$ is a deterministic unit vector with exponentially decaying coordinates. In the covariance model we set $\bSig=\bM$, whereas in the correlation model we standardize $\bM$ to unit diagonal and use the corresponding population correlation matrix
$
\bR=\diag^{-1/2}(\bM)\bM\diag^{-1/2}(\bM).
$
To make the results comparable across dimensional settings, we calibrate the leading population spike to remain at a common target level as $p$ varies. Thus the reported changes in empirical size mainly reflect finite-sample behavior rather than systematic changes in signal strength. For simplicity and interpretability, all simulations in this section are conducted under the single-spike setting $K=1$.

\subsection{One-sample inference for a benchmark principal direction}

We first study the one-sample benchmark-direction test. We consider three representative dimensional settings,
$
(p,n)\in\{(60,120),(120,240),(240,480)\},
$
under both the equicorrelated and Toeplitz-type designs, and under both Gaussian and Gamma$(4,0.5)-2$ innovations. For the power study, we fix $(p,n)=(120,240)$ and generate alternatives by rotating the benchmark direction away from the true principal direction according to the deviation parameter in the simulation algorithm.

Table~\ref{tab:onesample-size} reports empirical size under the null, and Figure~\ref{fig:onesample-power} shows the corresponding power curves. The covariance-based procedure is well calibrated, whereas the correlation-based procedure shows larger finite-sample deviations in smaller samples but improves as $(p,n)$ increases.

\begin{table}[htbp]
\centering
\caption{Empirical size of the one-sample benchmark-direction test at the nominal $5\%$ level.}
\label{tab:onesample-size}
\resizebox{\textwidth}{!}{%
\begin{tabular}{llcccccc}
\hline
&  \multicolumn{1}{c}{$(p,n)$}&\multicolumn{2}{c}{$(60,120)$} & \multicolumn{2}{c}{$(120,240)$} & \multicolumn{2}{c}{$(240,480)$} \\
\cline{3-8}
Design & Dist. & Cov & Corr & Cov & Corr & Cov & Corr \\
\hline
Equicorr & Gaussian & 0.051 & 0.077 & 0.047 & 0.064 & 0.050 & 0.057 \\
         & Gamma   & 0.063 & 0.074 & 0.061 & 0.069 & 0.051 & 0.054 \\
Toeplitz & Gaussian & 0.064 & 0.023 & 0.053 & 0.034 & 0.051 & 0.043 \\
         & Gamma   & 0.068 & 0.025 & 0.055 & 0.032 & 0.056 & 0.041 \\
\hline
\end{tabular}%
}
\end{table}

\begin{figure}[htbp]
\centering
\includegraphics[width=.85\textwidth]{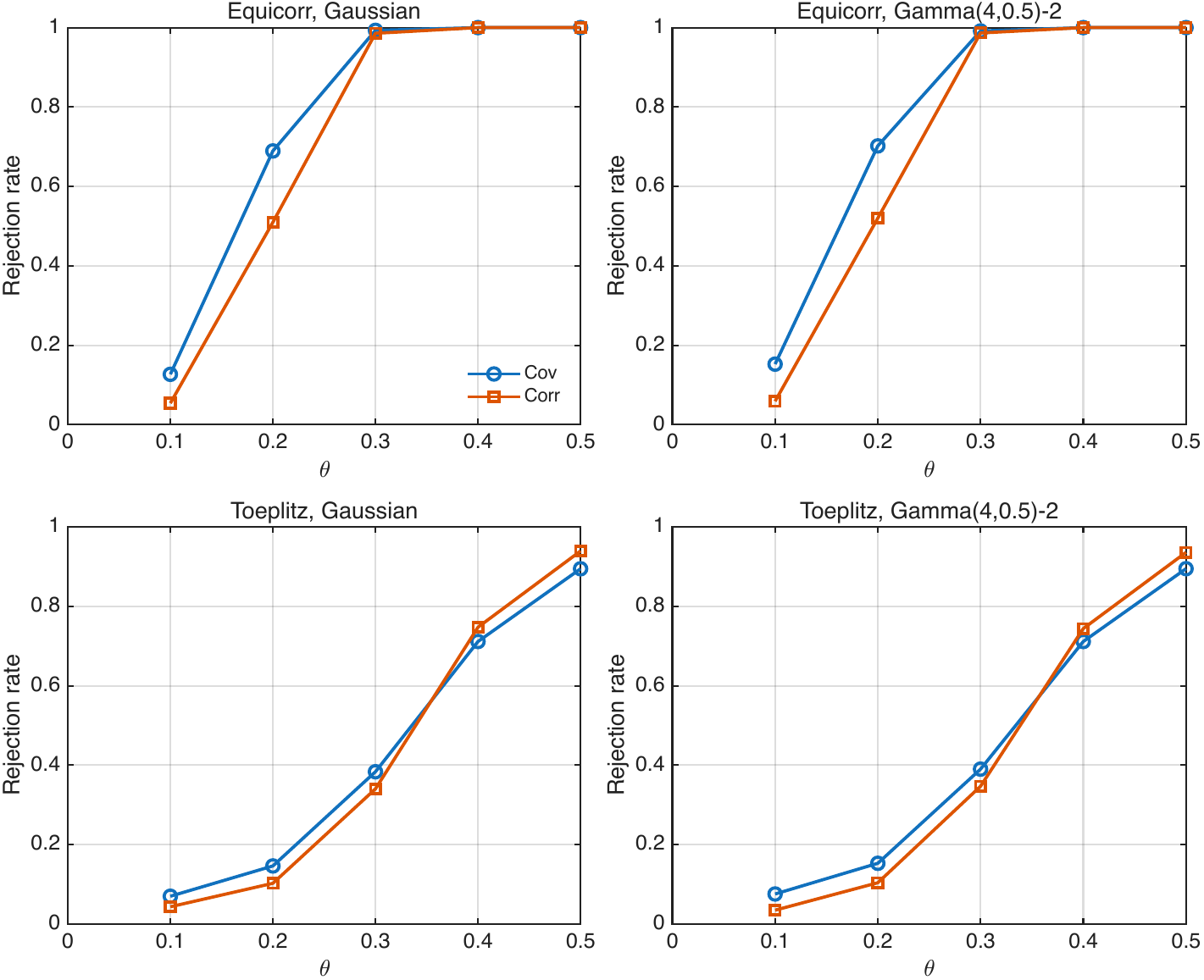}
\caption{Power curves of the one-sample benchmark-direction test for $(p,n)=(120,240)$ under the equicorrelated and Toeplitz-type designs.}
\label{fig:onesample-power}
\end{figure}

\subsection{Two-sample inference for a spiked eigenvalue}

We next study the two-sample spike test. For null calibration, we consider three representative settings,
\[
(p,n_1,n_2)\in\{(60,100,140),(120,200,280),(240,400,560)\},
\]
so that both sample sizes remain of the same order as $p$ but are not identical. The population designs and innovation distributions are taken to be the same as in the one-sample experiments, and the leading spikes under the null are calibrated in the same way across dimensional settings. For the power study, we focus on the representative configuration $(p,n_1,n_2)=(120,200,280)$.

The alternatives are generated by multiplying the second population spike by a factor in
\[
\{1.00,1.10,1.20,1.30,1.40,1.50,1.60,1.70,1.80\},
\]
while keeping the first population fixed. In the power analysis, we emphasize the correlation model.

Table~\ref{tab:twosample-size} reports empirical size under the null, and Figure~\ref{fig:twosample-power} shows the corresponding power curves. The covariance-based procedure is well calibrated across the reported settings, whereas the correlation-based version is less stable in smaller samples but approaches the nominal level as the sample sizes increase.

\begin{table}[htbp]
\centering
\caption{Empirical size of the two-sample spike test at the nominal $5\%$ level.}
\label{tab:twosample-size}
\resizebox{\textwidth}{!}{%
\begin{tabular}{llcccccc}
\hline
&  \multicolumn{1}{c}{$(p,n_1,n_2)$}& \multicolumn{2}{c}{$(60,100,140)$} & \multicolumn{2}{c}{$(120,200,280)$} & \multicolumn{2}{c}{$(240,400,560)$} \\
\cline{3-8}
Design & Dist. & Cov & Corr & Cov & Corr & Cov & Corr \\
\hline
Equicorr & Gaussian & 0.044 & 0.064 & 0.049 & 0.056 & 0.051 & 0.055 \\
         & Gamma   & 0.046 & 0.065 & 0.047 & 0.056 & 0.046 & 0.050 \\
Toeplitz & Gaussian & 0.039 & 0.071 & 0.049 & 0.065 & 0.047 & 0.051 \\
         & Gamma   & 0.039 & 0.073 & 0.044 & 0.060 & 0.047 & 0.058 \\
\hline
\end{tabular}%
}
\end{table}

\begin{figure}[htbp]
\centering
\includegraphics[width=.85\textwidth]{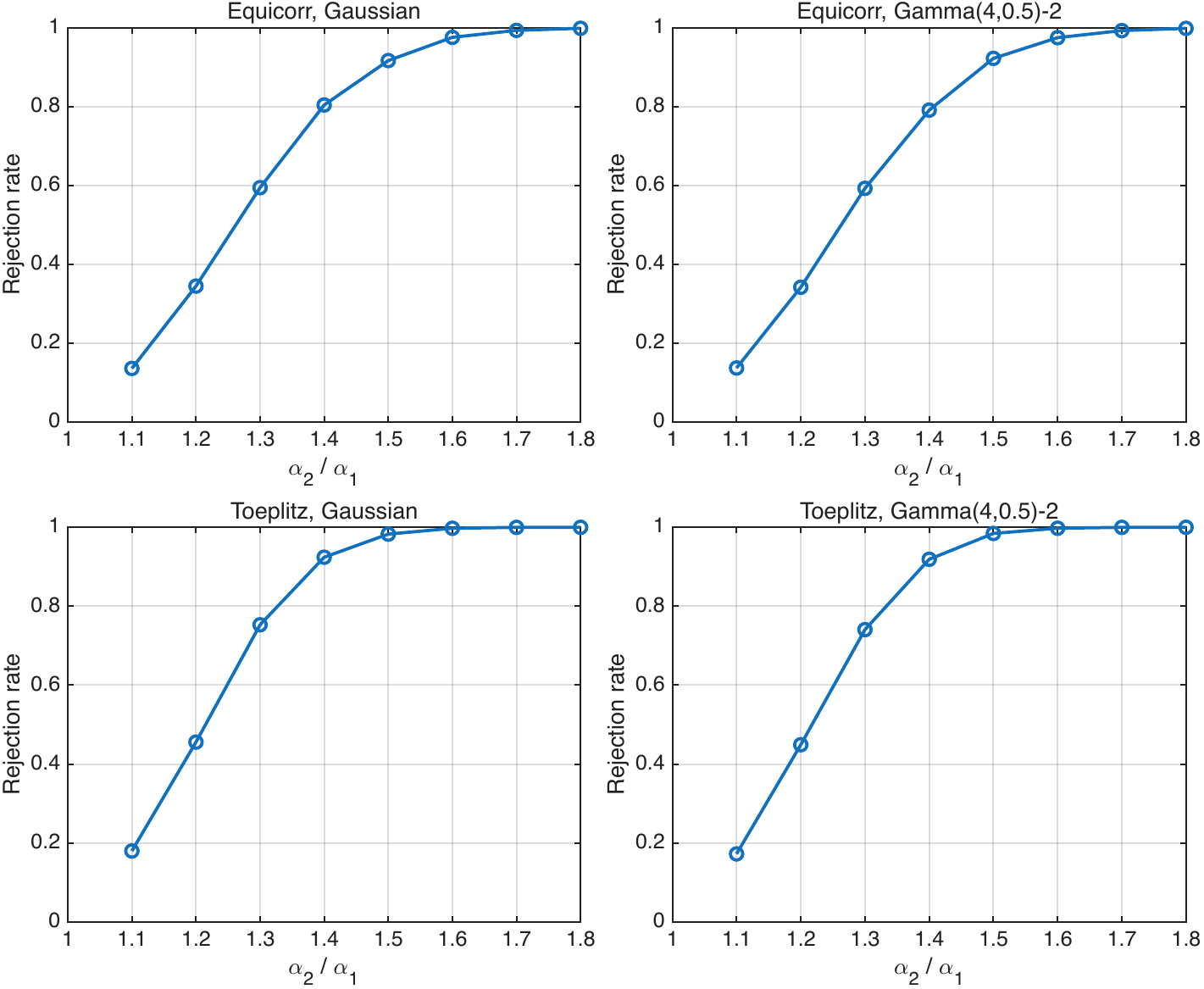}
\caption{Power curves of the two-sample spike test in the correlation model for $(p,n_1,n_2)=(120,200,280)$ under the equicorrelated and Toeplitz-type designs.}
\label{fig:twosample-power}
\end{figure}

\subsection{Real-data illustration}

We finally illustrate the proposed procedures on a publicly available panel of U.S.\ stock prices. Starting from a public five-year daily price dataset for a large set of S\&P 500 constituent stocks, we retain only those stocks with full trading-day coverage over the period from February 8, 2013 to February 7, 2018. We then rank the remaining stocks by average dollar trading volume and select the top 100 stocks to form a balanced high-dimensional panel. Daily close-to-close simple returns are computed from the adjusted closing prices.

This yields a $p\times n$ return matrix with $p=100$ and $n=1258$. To reduce the effect of heterogeneous marginal volatilities, we take correlation-based PCA as the primary analysis and report covariance-based results for comparison. This dataset provides a natural illustration of our methodology because equity returns typically contain a strong market-driven factor structure together with heterogeneous volatilities across assets. In such a setting, both principal-direction inference and comparison of leading spike strengths are statistically meaningful.

We consider two empirical questions that parallel the two inferential problems studied in Section~4. First, for the one-sample benchmark-direction problem, we test whether the leading population principal direction is consistent with the equally weighted market direction
$
v_0=p^{-1/2}\mathbf 1.
$
Second, for the two-sample spike comparison problem, we split the sample into two subperiods,
\[
\text{Period 1: February 11, 2013 -- July 31, 2015},\
\text{Period 2: August 3, 2015 -- February 7, 2018},
\]
with sample sizes
\[
n_1=623,\qquad n_2=635,
\]
and test
\[
H_0:\ \alpha_{1,\star,1}=\alpha_{2,\star,1},
\]
where $\star=\bSig$ or $\bR$.

The empirical findings are summarized in Table~\ref{tab:realdata}. For the one-sample benchmark-direction test, both covariance-based and correlation-based PCA strongly reject the null that the leading population principal direction coincides with the equally weighted market direction. Numerically, the estimated leading sample direction is still close to the benchmark, with
$
(v_0^\top \widehat z_1)^2=0.9685
$
for correlation PCA and
$
(v_0^\top \widehat z_{\bSig,1})^2=0.9104
$
for covariance PCA, but the deviations are statistically significant. This suggests that the dominant dependence direction is not simply the uniform market mode.

For the two-sample spike comparison, the conclusions differ between covariance-based and correlation-based PCA. In the correlation setting, the estimated leading spikes are
\[
\widehat\alpha_{1,\bR,1}=33.1178,\qquad
\widehat\alpha_{2,\bR,1}=33.3315,
\]
and the resulting test statistic is
\[
Z_{\bR,1}^{(2)}=-0.1282,
\]
with $p$-value $0.8980$, so the null is not rejected. By contrast, in the covariance setting,
\[
\widehat\alpha_{1,\bSig,1}=0.006631,\qquad
\widehat\alpha_{2,\bSig,1}=0.008496,
\]
and the test statistic is
\[
Z_{\bSig,1}^{(2)}=-3.0715,
\]
with $p$-value $0.0021$, leading to rejection of the null.

This illustration shows how the proposed procedures distinguish between changes in covariance-based and correlation-based factor strength in a real financial panel. These findings are consistent with the view that covariance-based PCA is more sensitive to changes in overall scale and volatility, whereas correlation-based PCA targets the standardized dependence structure more directly. In this dataset, the strength of the leading covariance factor changes significantly across the two periods, while the leading correlation factor remains comparatively stable.

\begin{table}[htbp]
\centering
\caption{Summary of the real-data illustration.}
\label{tab:realdata}
\small
\renewcommand{\arraystretch}{1.15}
\setlength{\tabcolsep}{4pt}
\begin{tabular}{llccc}
\hline
Method & Sample & Spike estimate & Benchmark test & Spike test \\
\hline
Corr-PCA 
& Full sample 
& $32.8178$ 
& reject ($p<10^{-16}$) 
& --- \\

Corr-PCA 
& Period 1 / Period 2 
& $33.1178,\ 33.3315$ 
& --- 
& no diff. ($p=0.8980$) \\

Cov-PCA  
& Full sample 
& $0.007494$ 
& reject ($p<10^{-16}$) 
& --- \\

Cov-PCA  
& Period 1 / Period 2 
& $0.006631,\ 0.008496$ 
& --- 
& diff. ($p=0.0021$) \\
\hline
\end{tabular}

\vspace{0.3em}
\begin{minipage}{0.92\textwidth}
\footnotesize
Benchmark test = one-sample benchmark-direction test; 
Spike test = two-sample spike comparison test.
\end{minipage}
\end{table}

\section{Discussion}\label{secdis}
In this paper, we developed a unified inferential framework for two statistical problems in high-dimensional PCA: inference for principal eigenspace functionals and comparison of leading spike strengths across populations. Under a generalized spiked model, we established the corresponding asymptotic theory for both the sample covariance matrix $\bS$ and the sample correlation matrix $\widehat{\bR}$.

On the covariance side, our results extend existing inferential theory beyond the identity-bulk setting. In particular, we derive asymptotic distributions for projector-type functionals of sample spiked eigenspaces and obtain feasible procedures for leading spike comparison.

On the correlation side, our analysis shows that normalization has a genuinely nontrivial second-order effect on high-dimensional PCA. Although the first-order limits of the spiked eigenstructure are still governed by the population correlation matrix, the second-order fluctuations differ essentially from those in the covariance setting. This distinction affects both spiked eigenvalues and eigenspace inference, showing that the choice between covariance- and correlation-based PCA has direct consequences for second-order statistical inference.

A key technical contribution is a joint central limit theorem for bilinear forms involving the resolvent of the sample covariance matrix and its diagonal entries. This result provides a unified route to the asymptotic analysis of both $\bS$ and $\widehat{\bR}$ and makes the correlation-side theory tractable under generalized bulk structure. It may also be useful in other random matrix problems involving random normalization or diagonal adjustment.

The present work suggests several directions for future research. One is to develop a more systematic calibration theory for correlation-based inference in non-Gaussian settings, where normalization introduces additional nuisance quantities. Another is to extend the inferential analysis to weaker or near-critical spikes, thereby clarifying the boundary between detectability and reliable eigenspace recovery. It would also be of interest to apply the projector-type perspective developed here to related problems such as eigenspace comparison, change-point analysis for principal subspaces, and factor-space inference in high-dimensional latent structure models.

Overall, the results of this paper show that spiked eigenstructure inference provides a natural perspective on high-dimensional PCA. By treating covariance-based and correlation-based PCA within a common generalized spiked framework, we obtain a clearer understanding of how signal strength, bulk structure, and normalization jointly determine the asymptotic behavior of sample principal components and the resulting statistical procedures.

\section*{Supplementary Material}

The supplementary material, entitled
\textit{Supplement to ``Inference for Spiked Eigenstructure under Generalized Covariance and Correlation Models''},
contains technical proofs, auxiliary lemmas, explicit variance formulas, and supporting results for the plug-in procedures.

%\bibliographystyle{apalike}
%\bibliography{library}

\end{document}